\documentclass[12pt,reqno]{amsart}
\usepackage{latexsym,amsmath,amsfonts,amssymb,amsthm}
\def\Z{\mathbb Z}
\def\cZ{\mathcal Z}
\def\N{\mathbb N}
\def\T{\mathbb T}
\def\Q{\mathbb Q}
\def\R{\mathbb R}

\def\E{\mathbb E}
\def\1{{\mathbf 1}}

\def\K{\mathcal K}

\def\cl{{\rm cl}}

\def\cG{\mathcal{G}}
\def\cB{\mathcal{B}}
\def\fB{{\mathfrak B}}\def\fD{{\mathfrak D}}
\def\cC{\mathcal{C}}

\theoremstyle{plain}
\newtheorem{Thm}{Theorem}
\newtheorem{Lem}{Lemma}

\theoremstyle{definition}

\theoremstyle{remark}

\begin{document}
\title{Note on polynomial recurrence}
\author{Hao Pan}
\email{haopan79@zoho.com}
\address{Department of Mathematics, Nanjing University,
Nanjing 210093, People's Republic of China}
\begin{abstract}
Let $(X,\mu,T_1,\ldots,T_l)$ be a measure-preserving system with those $T_i$ are commuting. 
Suppose that the polynomials $p_1(t),\ldots,p_{l}(t)\in\Z[t]$ with $p_j(0)=0$ have distinct degrees.
Then for any $\epsilon>0$ and $A\subseteq X$ with $\mu(A)>0$, the set
$$
\{n:\,\mu(A\cap T_1^{-p_1(n)}A\cap\cdots\cap T_l^{-p_l(n)}A)\geq\mu(A)^{l+1}-\epsilon\}
$$ 
has bounded gaps.
\end{abstract}

\maketitle

Let $(X,\fB,\mu,T)$ be a measure-preserving system. The well-known Poincare recurrence theorem asserts for any $A\in\fB$ with $\mu(A)>0$, there exist infinitely many $n\in\N$ such that $\mu(A\cap T^{-n}A)>0$.
For a subset $S$ of $\N$, we say $S$ is {\it syndetic} if there exists a constant $C>0$ such that
$
S\cap [x,x+C]\neq\emptyset
$
for any sufficiently large $x$, i.e., $S$ has bound gaps. In \cite{K}, Khintchine proved that for any $A\in\fB$ with $\mu(A)>0$ and any $\epsilon>0$, the set
$$
\{n\in\N:\,\mu(A\cap T^{-n}A)\geq\mu(A)^2-\epsilon\}
$$
is syndetic.

In \cite{BHK}, Bergelson, Host and Kra proved that for any $A\in\fB$ with $\mu(A)>0$ and any $\epsilon>0$, if $T$ is ergodic, then both
$$
\{n\in\N:\,\mu(A\cap T^{-n}A\cap T^{-2n}A)\geq\mu(A)^3-\epsilon\}
$$
and 
$$
\{n\in\N:\,\mu(A\cap T^{-n}A\cap T^{-2n}A\cap T^{-3n}A)\geq\mu(A)^4-\epsilon\}
$$
are syndetic. However, they also showed that $\{n\in\N:\,\mu(A\cap\cdots\cap T^{-kn}A)\geq\mu(A)^{k+1}-\epsilon\}$
is maybe not syndetic when $k\geq 4$.
Subsequently, Frantzikinakis \cite{F} obtained the polynomial extensions of the above results. He showed that if  
$T$ is ergodic and $p_1(t),p_2(t)\in\Z[t]$ with $p_i(0)=0$,
then for any $A\in\fB$ with $\mu(A)>0$ and any $\epsilon>0$
$$
\{n\in\N:\,\mu(A\cap T^{-p_1(n)}A\cap T^{-p_2(n)}A)\geq\mu(A)^3-\epsilon\}
$$
is syndetic. Furthermore, for most of the triples $(p_1(t),p_2(t),p_3(t))\in\Z[t]^3$ with $p_i(0)=0$, 
$$
\{n\in\N:\,\mu(A\cap T^{-p_1(n)}A\cap T^{-p_2(n)}A\cap T^{-p_3(n)}A)\geq\mu(A)^4-\epsilon\}
$$
is also syndetic.

On the hand, in \cite{CFH}, Chu, Frantzikinakis and Host proved that if $(X,\mu,T_1,\ldots,T_l)$ be a measure-preserving system with those $T_i$ are commuting, then for distinct $n_1,\ldots,n_l\in\N$ and $\epsilon>0$,
$$
\{n:\,\mu(A\cap T_1^{-n_1^{d_1}}A\cap\cdots\cap T_l^{-n_l^{d_l}}A)\geq\mu(A)^{l+1}-\epsilon\}
$$ 
is syndetic.
Notice that the result of Chu, Frantzikinakis and Host doesn't requirement those $T_i$ are ergodic. Furthermore, as they mentioned, their arguments are also valid for 
$$
\{n:\,\mu(A\cap T_1^{-p_1(t)}A\cap\cdots\cap T_l^{-p_l(t)}A)\geq\mu(A)^{l+1}-\epsilon\},
$$ 
where $p_1(t),\ldots,p_l(t)\in\Z[t]$ satisfy $p_1(0)=0$ and $t^{\deg p_j+1}\mid p_{j+1}(t)$ for $1\leq j<l$.
And they also conjecture that the same assertion should still hold for the polynomials $p_1(t),\ldots,p_l(t)\in\Z[t]$ with $p_j(0)=0$ having distinct degrees.

In the note, we try to give an affirmative answer for this problem.
\begin{Thm}\label{main}
Let $(X,\mu,T_1,\ldots,T_l)$ be a measure-preserving system with those $T_i$ are commuting. 
Suppose that the polynomials $p_1(t),\ldots,p_{l}(t)\in\Z[t]$ with $p_j(0)=0$ have distinct degrees.
Then for any $\epsilon>0$ and $A\subseteq X$ with $\mu(A)>0$, the set
$$
\{n:\,\mu(A\cap T_1^{-p_1(n)}A\cap\cdots\cap T_l^{-p_l(n)}A)\geq\mu(A)^{l+1}-\epsilon\}
$$ 
is syndetic.
\end{Thm}
In view of the well-known Furstenberg correspondence principle, we have the following combinatorial consequence.
\begin{Thm}\label{mainc}
Let $\Lambda\subseteq\Z^k$ with $\overline{d}(\Lambda)>0$ and $\vec{v}_1,\ldots,\vec{v}_l$ be vectors in $\Z^k$.
Suppose that the polynomials $p_1(t),\ldots,p_{l}(t)\in\Z[t]$ with $p_j(0)=0$ have distinct degrees.
Then for any $\epsilon>0$ and $A\subseteq X$ with $\mu(A)>0$, the set
$$
\{n:\,\overline{d}(\Lambda\cap(p_1(n)\vec{v}_1+\Lambda)\cap\cdots\cap (p_l(n)\vec{v}_l+\Lambda)\geq\overline{d}(\Lambda)^{l+1}-\epsilon\}
$$ 
is syndetic.
\end{Thm}
We shall prove Theorem \ref{main} via several auxiliary lemmas.
For a polynomial $f(x)$, let $[x^d]f(x)$ denote the coefficient of $x^d$ in $f(x)$.
\begin{Lem}\label{cofr}
Suppose that $\alpha$ is irrational and $f(x)\in\Q[x]$. Let $q(x)$ be a polynomial such that $[x^k]q$ is irrational for some $k\geq 1$. Then there exists $d\geq\deg f$ such that
$$
[x^d]\big(\alpha f(x)+f(x)q(x)\big)
$$
is irrational. In particular, for any $q(x)\in\R[x]$, there exists $d\geq\deg f$ such that
$$
[x^d]\big(\alpha f(x)+f(x)^2q(x)\big)\not\in\Q.
$$
\end{Lem}
\begin{proof}
Let $d_0=\deg f$. If $[x^d]\big(\alpha f(x)+f(x)q(x)\big)$ is rational, since $\alpha\not\in\Q$ and $f(x)\in\Q[x]$, 
we must have $q(x)\not\in\Q[x]$. Let
$$
d_1=\max\{k:\,[x^{k}]q(x)\text{ is irrational}\}.
$$\
It is easy to see that
$$
[x^{d_0+d_1}]f(x)q(x)-[x^{d_0}]f(x)\cdot [x^{d_1}]q(x)=\sum_{j\geq 1}[x^{d_0-j}]f(x)\cdot[x^{d_1+j}]q(x)\in\Q.
$$
Then noting that $d_1\geq 1$, we have
$$
[x^{d_0+d_1}]\big(\alpha f(x)+f(x)q(x)\big)=[x^{d_0+d_1}]f(x)q(x)\in[x^{d_0}]f(x)\cdot [x^{d_1}]q(x)+\Q
$$
is irrational.

Next, we shall show that for any $q(x)\in\R[x]$, $
[x^k]\big(\alpha f(x)+f(x)^2q(x)\big)\not\in\Q
$ for some $k\geq\deg f$. There is nothing to do when $q(x)\in\Q[x]$, since $\alpha\not\in\Q$.
If $q(x)\not\in\Q[x]$, by the above discussions, we also have $[x^{\deg f+d_1}]f(x)q(x)$ is irrational, where
$d_1$ is the largest integer such that $[x^{d_1}]q(x)$ is irrational. Applying the first assertion of this lemma, we get the desired result.
\end{proof}
For a compact Lie group $G$, let $m_{G}$ denote the unique normal Haar measure on $G$.
\begin{Lem}\label{tted}
Let $m,l\in\N$ and $T:\,\T^{m}\to\T^{m}$ be an ergodic unipotent affine transformation. Suppose that
$p_0(t)\in\Z[t]$, $p_1(t),\ldots,p_{l}(t)\in\R[t]$ and $\deg p_0>\max_{j\geq 1}\deg p_j$. Suppose that
$\tilde{p}(n)=(p_1(n),\ldots,p_{l}(n))$ modulo $1$
is equidistributed on $\T^{l}$. Then for $m_{T^{m}}$-almost (only depending on $T$) every $x\in\T^{m}$, the sequence $(T^{p_0(n)}x,\tilde{p}(n))_{n\in\N}$ is 
equidistributed on $\T^{m}\times\T^{l}$.
\end{Lem}
\begin{proof}
Suppose that $Tx=Sx+b$ where $S$ is a unipotent homomorphism of $\T^{m}$ and $b\in\T^{m}$.
Clearly we may only consider those $x\in\T^{m}\setminus(b\cdot\Q^m)$.
Suppose that $x_0\in\T^{m}\setminus(b\cdot\Q^m)$. By Weyl's equidistribution theorem, we only need to show that
$$
\lim_{N\to+\infty}\frac1N\sum_{n=1}^N\chi_1(T^{p_0(n)}x_0)\chi_2(\tilde{p}(n))=0
$$
for any non-trivial character $\chi=(\chi_1,\chi_2)$ of $\T^{m}\times\T^{l}$.
If $\chi_1=1$, since $\tilde{p}(n)$ modulo $1$ is equidistributed on $\T^{l}$, we have
$$\lim_{N\to+\infty}\frac1N\sum_{n=1}^N\chi_2(\tilde{p}(n))=0.$$
Assume that $\chi_1\neq 1$. Since $T$ is ergodic unipotent affine transformation, according
to the discussions of \cite[Lemma 7.3]{CFH},
we have
$$
T^n x_0=x_0+n((S-I)x_0+b)+n^2\tilde{q}(n)
$$
for some $\tilde{q}(t)=(q_1(t),\ldots,q_{m_1}(t))\in\R[t]^{m_1}$ 
and
$$
\chi_1((S-I)x_0+b)=e(\alpha)
$$
for some irrational $\alpha$, where $e(x)=\exp(2\pi\sqrt{-1}x)$.

Assume that $\chi_1(x)=e(\beta_1\cdot x)$ and $\chi_2(x)=e(\beta_2\cdot x)$ where $\beta_1\in\T^m$ and $\beta_2\in\T^l$.
Then
$$
\chi_1(T^{p_0(n)}x_0)\cdot\chi_2(\tilde{p}(n))=e\big(\beta_1\cdot x_0+p_0(n)\alpha+p_0(n)^2(\beta_1\cdot \tilde{q}(n))+\beta_2\cdot \tilde{p}(n)\big).
$$
It follows from Lemma \ref{cofr}, there exists $d\geq \deg p_0$ such that
$$
[t^d]\big(p_0(t)\alpha+p_0(t)^2(\beta_1\cdot\tilde{q}(t))\big)
$$
is irrational. 
Noting that $$\deg\big(\beta_2\cdot \tilde{p}(t)\big)\leq\max_{1\leq j\leq l}\{\deg p_j\}<\deg p_0,$$
we know the polynomial
$$
\beta_1\cdot x_0+p_0(t)\alpha+p_0(t)^2(\beta_1\cdot \tilde{q}(t))+\beta_2\cdot\tilde{p}(t)
$$
has at least one non-constant-term coefficient is irrational.
Applying a well-known result of Weyl, we get
\begin{align*}
&\lim_{N\to+\infty}\frac1N\sum_{n=1}^N\chi_1(T^{p_0(n)}x_0)\chi_2(\tilde{p}(n))\\
=&\lim_{N\to+\infty}\frac1N\sum_{n=1}^Ne\big(\beta_1\cdot x_0+p_0(n)\alpha+p_0(n)^2(\beta_1\cdot \tilde{q}(n))+\beta_2\cdot \tilde{p}(n)\big)=0.
\end{align*}
\end{proof}

Here we introduce some notions on nilmanifolds and nilsequences. Suppose that $G$ is a nilpotent Lie group. If the $(k+1)$-th commutator group of $G$ is trivial but the $k$-th is not, then we
say $G$ is a {\it $k$-step nilpotent group}.
Let $\Gamma$ be a discrete co-compact subgroup of $G$. We call $G/\Gamma$
a {\it nilmanifold}. Furthermore, if $G$ is a $k$-step nilpotent group,
we say $G/\Gamma$ is a {\it $k$-step nilmanifold}.

Let a $k$-step nilpotent group $G$ act on the nilmanifold $G/\Gamma$ by left translation, i.e., $T_a(g\Gamma)=(ag)\Gamma$ for a fixed $a\in G$.
Let $\cG/\Gamma$ be the Borel $\sigma$-algebra of $G/\Gamma$. Then for any $a\in G$, we call $(G/\Gamma,\cG/\Gamma,m_{G/\Gamma},T_a)$ a {\it $k$-step system}. Let $a\in G$, $x\in G/\Gamma$ and $f\in\cC(G/\Gamma)$, where $\cC(X)$ denotes the set of all continuous functions on $X$. Then we say $(f(a^nx))_{n\in\N}$ is a {\it basis $k$-step nilsequence}. And a union limit of basis $k$-step nilsequences is called a {\it $k$-step nilsequence}.
\begin{Lem}\label{XYed}
Let $X=G/\Gamma$ be a connected nilmanifold and $a$ be an ergodic element of $G$. Suppose that
$p_0(t),p_1(t),\ldots,p_{l}(t)\in\Z[t]$ with $p_0(0)=0$ and $\deg p_0>\max_{j\geq 1}\deg p_j$. 
Let $Y=H/\Delta$ be a nilmanifold and $g(n)=a_1^{p_1(n)}\cdots a_{l}^{p_l(n)}$, where $a_1,\ldots,a_l\in Y$.
Then there exists an $r_0\in\N$ such that for $m_X$-almost (only depending on $a$) every $x\in X$ such that
for any $r\in r_0\N$, $(a^{p_0(n)}x,g(rn)y)_{n\in\N}$ is equidistributed on $X\times \cl_Y(\{g(n)y:\,n\in\N\})$.
\end{Lem}
\begin{proof}
First, assume that $\cl_Y(\{g(n)y:\,n\in\N\})$ is connected. 
And without loss of generality, we may assume $Y=\cl_Y(\{g(n)y:\,n\in\N\})$.
According to the discussions in the proof of \cite[Lemma 7.6]{CFH}, 
our problem can be reduced to the following special case:\medskip

\noindent $X=\T^{m_1}$ and $Y=\T^{m_2}$. And $T_a:\,x\to ax$ on $X$
and $T_{a_j}:\,y\to a_jy$ on $Y$ are respectively
unipotent affine transformations. Furthermore, $T_a$ is ergodic on $X$ and $(T_{a_1}^{p_1(n)}\cdots T_{a_l}^{p_l(n)}y)_{n\in\N}$ is 
equidistributed on $Y$.\medskip

Note that each coordinate of $T_{a_1}^{p_1(rn)}\cdots T_{a_l}^{p_l(rn)}y$ is a polynomial in $n$ with the degree at most $$\max_{1\leq j\leq l}\deg p_j<\deg p_0.$$
By Lemma \ref{tted}, for $m_X$-almost $x\in X$ and each $r\in\N$, $$
(T_{a}^{p_0(rn)}x,T_{a_1}^{p_1(rn)}\cdots T_{a_l}^{p_l(rn)}y)
$$
is equidistributed on $X\times Y$. Thus the assertion of this lemma holds for connected $\cl_Y(\{g(n)y:\,n\in\N\})$.

Next, assume that $\cl_Y(\{g(n)y:\,n\in\N\})$ is not connected.
Then by \cite[Theorem 7.1 (i)]{CFH}, $\hat{Y}:=\cl_Y(\{g(r_0n)y:\,n\in\N\})$ is connected and
$(\{g(r_0n)y\})_{n\in\N}$ is equidistributed on $\hat{Y}$
for some $r_0\in\N$. 
Let $\hat{p}_0(t)=p_0(r_0t)/r_0$, $\hat{a}=a^{r_0}$ and $\hat{p}_j(t)=p_j(r_0t)$. 
Note that $\hat{a}$ is also an ergodic element of $G$. Then for each $s\in\N$, by the above discussions for the connected case, we can get that
$$
(a^{p_0(r_0sn)}x,a_1^{p_1(r_0sn)}\cdots a_l^{p_l(r_0sn)}y)_{n\in\N}
=(\hat{a}^{\hat{p}_0(sn)}x,a_1^{\hat{p}_1(sn)}\cdots a_l^{\hat{p}_l(sn)}y)_{n\in\N}
$$
is equidistributed on $X\times\hat{Y}$.
\end{proof}
For a measure-preserving system $(X,\fB,\mu,T)$, if $\fD$ is a $T$-invariant sub-$\sigma$-algebra of $\fB$ and $f\in L^1(\mu)$, then let $\E_\mu(f|\fD)$ denote the conditional expectation of $f$ with respect to $\fD$. 
We know that
$$
\int\E_\mu(f|\fD)d\mu=
\int fd\mu.
$$
Furthermore, we say $f\perp\fD$ provided $\E_\mu(f|\fD)=0$.

Let $\K_{rat}(T)$ denote the rational Kronecker factor of $T$, i.e., $\K_{rat}(T)$ is spanned by
$$
\{f\in L^\infty(\mu):\,T^df=f\text{ for some }d\in\N\}.
$$
Clearly $\K_{rat}(T)$ is $T$-invariant.
\begin{Lem}\label{fsz}
Let $X=G/\Gamma$ be a nilmanifold and $a$ be an ergodic element of $G$. 
Suppose that $f\in \cC(X)$ and $f\perp\K_{rat}(T_a)$.
Let $(u_{1,n})_{n\in\N},\ldots,(u_{l,n})_{n\in\N}$ be finite nilsequences.
Suppose that
$p_0(t),p_1(t),\ldots,p_{l}(t)\in\Z[t]$ with $p_0(0)=0$ and $\deg p_0>\max_{j\geq 1}\deg p_j$. 
Then there exists an $r_0\in\N$ such that for $m_X$-almost every $x\in X$ such that
$$
\lim_{N-M\to\infty}\frac1{N-M}\sum_{n=M}^{N-1}f(a^{p_0(rn)}x)\cdot u_{1,p_1(rn)}\cdots u_{l,p_1(rn)}=0
$$
for each $r\in r_0\N$. Furthermore, the set of full $m_X$-measure can only depend on $a$.
\end{Lem}
\begin{proof}
Assume that $X$ is connected. By an approximation argument, we may assume that each $(u_{i,n})_{n\in\N}$
is a basic finite step nilsequence. Then there exist nilmanifolds $X_i=G_i/\Gamma_i$, elements $a_i\in G_i$ and functions $f_i\in \cC(X_i)$ such that
$u_{i,n}=f_i(a_i^n\Gamma_i)$.

Let
$\tilde{G}=G_1\times\cdots\times G_l$, $\tilde{\Gamma}=\Gamma_1\times\cdots\Gamma_l$
and $\tilde{X}=X_1\times\cdots\times X_l$.
Let $\tilde{a}_i=(0,\ldots,a_i,\ldots,0)$ and $g(n)=\tilde{a}_1^{p_1(n)}\cdots\tilde{a}_l^{p_l(n)}$. By Lemma \ref{XYed}, there exists $r_0\in\N$
such that for any $r\in r_0\N$,
$(a^{p_0(rn)}x,g(rn)\tilde{\Gamma})_{n\in\N}$ is equidistributed on $X\times Y$,
where
$Y=\cl_{\tilde{X}}(g(r_0n)\tilde{\Gamma}:\,n\in\N)$.
So letting $F=f_1\cdots f_l$ and recalling $\E_{m_X}(f|\K_{rat}(T_a))=0$, we have
\begin{align*}
&\lim_{N-M\to\infty}\frac1{N-M}\sum_{n=M}^{N-1}f(a^{p_0(rn)}x)\cdot F(g(rn)\tilde{\Gamma})\\
=&\int fd m_{X}\cdot\int Fd m_{\tilde{X}}=\int \E_{m_X}(f|\K_{rat}(T_a))d m_{X}\cdot\int Fd m_{\tilde{X}}=0.
\end{align*}

Suppose that $X$ is not connected and $X_0$ is a connected component of $X$. Since $a$ is ergodic, there exists $k_0\in\N$ such that $X$ is the disjoint union of $X_i=a^iX_0$, $j=1,\ldots,k_0$. 
Let $\hat{a}=a^{k_0}$, $\hat{p}_0(t)=p_0(k_0t)/k_0$ and $\hat{g}(t)=g(k_0t)$ for $j\geq 1$.
For each $i$, since $X_i$ is connected and $\E_{m_{X_i}}(f|\K_{rat}(T_{\hat{a}}))=0$, there exists $r_i\in\N$ such that
for any $r\in r_i\N$,
$$
\lim_{N-M\to\infty}\frac1{N-M}\sum_{n=M}^{N-1}f(\hat{a}^{\hat{p}_0(rn)}x)\cdot F(\hat{g}(rn)\tilde{\Gamma})=0
$$
for $m_{X_i}$-almost $x\in X_i$. So for any $r\in r_1\ldots r_l\N$ and $m_{X}$-almost $x\in X$, we have
\begin{align*}
&\lim_{N-M\to\infty}\frac1{N-M}\sum_{n=M}^{N-1}f(a^{p_0(k_0rn)}x)\cdot F(g(k_0rn)\tilde{\Gamma})\\
=&
\lim_{N-M\to\infty}\frac1{N-M}\sum_{n=M}^{N-1}f(\hat{a}^{\hat{p}_0(rn)}x)\cdot F(\hat{g}(rn)\tilde{\Gamma})=0.
\end{align*}
\end{proof}

For a measure-preserving system $(X,\fB,\mu,T)$, let
$\cZ_{k,T}$ be the $T$-invariant factor constructed by Host and Kra in \cite{HK}.
Here we won't give the explicit construction of $\cZ_{k,T}$, and the readers may refer to \cite{HK} for the related details.
\begin{Lem}\label{fse}
Let $(X,\mu,T_1,\ldots,T_l)$ be a system. Suppose that
$p_1(t),\ldots,p_{l}(t)\in\Z[t]$ with $p_j(0)=0$ and $\deg p_1>\max_{j\geq 2}\deg p_j$. 
Suppose that $f_1,\ldots,f_l\in L^\infty(\mu)$ and $f_1\perp\K_{rat}(T_1)$. Then for any $\epsilon>0$, there exists $r_0\in\N$ such that for each $r\in r_0\N$,
$$
\lim_{N-M\to\infty}\bigg\|\frac1{N-M}\sum_{n=M}^{N-1}T_1^{p_1(rn)}f_1\cdots
T_l^{p_l(rn)}f_l\bigg\|_2\leq\epsilon.
$$
\end{Lem}
\begin{proof}
Without loss of generality, we may assume that $\|f_i\|_\infty\leq1$. 
In \cite[Theorem 1.2]{CFH}, Chu, Frantzikinakis and Host proved that
there exists $k\in\N$ (only depending on $\max_{j}\deg p_j$ and $l$) such that
if $f_i\perp \cZ_{k,T_i}$ for some $i$, then
\begin{equation}
\lim_{N-M\to\infty}\bigg\|\frac1{N-M}\sum_{n=M}^{N-1}T_1^{p_1(rn)}f_1\cdots
T_l^{p_l(rn)}f_l\bigg\|_2=0.
\end{equation}
So we may assume that $f_i\in L^\infty(\cZ_{k,T_i},\mu)$ for each $i$.
By \cite[Proposition 3.1]{CFH}, for each $2\leq i\leq l$,
there exists $\tilde{f}_i$ with $\|\tilde{f}_i\|_\infty\leq $ such that

\medskip\noindent(i) $\tilde{f}_i\in  L^\infty(\cZ_{k,T_i},\mu)$ and $\|f_i-\tilde{f}_i\|_2\leq \epsilon/(2l+2)$;

\medskip\noindent(ii) For $\mu$-almost $x\in X$ and every $r\in\N$, $(\tilde{f}_i(T_i^{p_i(rn)}x))_{n\in\N}$ is a $(k\deg p_i)$-step nilsequence.

\medskip
Now it suffices to show that if $f_1\perp \cZ_{k,T_1}$ and $\tilde{f}_i\in L^\infty(\cZ_{k,T_i},\mu)$ for $2\leq i\leq l$, then there exists $r_0\in\N$ such that for any $r\in r_0\N$,
\begin{equation}\label{ft}
\lim_{N-M\to\infty}\bigg\|\frac1{N-M}\sum_{n=M}^{N-1}T_1^{p_1(rn)}f_1\cdot
T_2^{p_2(rn)}\tilde{f}_2\cdots
T_l^{p_l(rn)}\tilde{f}_l\bigg\|_2\leq\frac{\epsilon}2.
\end{equation}

First, assume that $T_1$ is ergodic.
According to the discussions in the proof of \cite[Lemma 7.8]{CFH}, 
we can reduce the proof of (\ref{ft}) to a special case:
$T_1:\,x\to ax$ is an ergodic rotation on a nilmanifold $X=G/\Gamma$ and
$f_1\in\cC(X)$ satisfies $f_1\perp\K_{rat}(T_1)$. Now by (ii) and Lemma \ref{fsz}, for each $r\in r_0\N$,
we have
$$
\lim_{N-M\to\infty}\frac1{N-M}\sum_{n=M}^{N-1}f_1(T_1^{p_1(rn)}x)\cdot
\tilde{f}_2(T_2^{p_2(rn)}x)\cdots
\tilde{f}_l(T_l^{p_l(rn)}x)=0
$$
for $\mu$-almost $x\in X$.
It follows that
$$
\lim_{N-M\to\infty}\bigg\|\frac1{N-M}\sum_{n=M}^{N-1}T_1^{p_1(rn)}f_1\cdot
T_2^{p_2(rn)}\tilde{f}_2\cdots
T_l^{p_l(rn)}\tilde{f}_l\bigg\|_2=0<\frac{\epsilon}2.
$$

When $T_1$ is not ergodic, by the discussions in the proof of \cite[Lemma 7.8]{CFH}, (\ref{ft})
also can be derived via the ergodic decomposition of $\mu$. 
\end{proof}
\begin{Lem}\label{fTr}
Let $(X,\fB,\mu,T_1,\ldots,T_l)$ be a measure-preserving system with those $T_i$ are commuting.
Suppose that
$p_1(t),\ldots,p_{l}(t)\in\Z[t]$ with $p_j(0)=0$ have distinct degrees. 
Suppose that $f_1,\ldots,f_l\in L^\infty(\mu)$ and $f_i\perp\K_{rat}(T_1)$ for some $1\leq i\leq l$. Then for any $\epsilon>0$, there exists $r_0\in\N$ such that for each $r\in r_0\N$,
$$
\lim_{N-M\to\infty}\bigg\|\frac1{N-M}\sum_{n=M}^{N-1}T_1^{p_1(rn)}f_1\cdots
T_l^{p_l(rn)}f_l\bigg\|_2\leq\epsilon.
$$
\end{Lem}
\begin{proof}
Without loss of generality, we may assume that $\deg p_1>\deg p_2>\cdots>\deg p_l$ and $\|f_j\|_\infty\leq 1$.
We use an induction on $l$. The case $l=1$ is known. Suppose that $l\geq 2$. 

If $f_1\perp\K_{rat}(T_1)$, then the lemma immediately follows from Lemma \ref{fsz}.
So we may assume that $f_i\perp \K_{rat}(T_i)$ for some $i\geq 2$ and $f_1$ is $\K_{rat}(T_1)$-measurable.
Furthermore, by an approximation argument, we may assume that $f_1$
is $\K_{r_1}(T_1)$-measurable for some $r_1\in\N$, i.e., $T^{r_1}f_1=f_1$. Then
there exists $r_2\in r_1\N$ such that for any $r\in r_2\N$
\begin{align*}
&\lim_{N-M\to\infty}\bigg\|\frac1{N-M}\sum_{n=M}^{N-1}T_1^{p_1(rn)}f_1\cdot T_2^{p_2(rn)}f_2\cdots
T_l^{p_l(r_2n)}f_l\bigg\|_2\\
=&\lim_{N-M\to\infty}\bigg\|f_1\cdot\frac1{N-M}\sum_{n=M}^{N-1}T_2^{p_2(rn)}f_2\cdots
T_l^{p_l(rn)}f_l\bigg\|_2\\
\leq&\lim_{N-M\to\infty}\bigg\|\frac1{N-M}\sum_{n=M}^{N-1}T_2^{p_2(rn)}f_2\cdots
T_l^{p_l(rn)}f_l\bigg\|_2\leq\epsilon,
\end{align*}
where we used the induction hypothesis in the last step. 
\end{proof}
The following lemma is \cite[Lemma 1.6]{C}.
\begin{Lem}\label{EfX}
Suppose that $(X,\cB,\mu)$ is a probability space and $X_1,\ldots,X_l$ are sub-$\sigma$ algebras of $X$. Then for 
any non-negative $f\in L^2(\mu)$,
$$
\int f\cdot\E(f|X_1)\cdots\E(f|X_l)d\mu\geq\bigg(\int fd\mu\bigg)^{l+1}.
$$
\end{Lem}
Now we are ready to prove Theorem \ref{main}
\begin{proof}[Proof of Theorem \ref{main}]
By Lemma \ref{fTr}, there exists $r_0\in\N$ such that for each $r\in r_0\N$,
$$
\lim_{N-M\to\infty}\bigg\|\frac1{N-M}\sum_{n=M}^{N-1}T_1^{p_1(rn)}f_1\cdots
T_l^{p_l(rn)}f_l\bigg\|_2\leq\frac{\epsilon}{2^{l+1}}
$$
if some $f_j\perp\K_{rat}(T_j)$. Since $\1_A-\E(\1_A|\K_{rat}(T_j))$ orthogonal to $\K_{rat}(T_j)$, we can get
\begin{align*}
&\lim_{N-M\to\infty}\bigg\|\frac1{N-M}\sum_{n=M}^{N-1}\1_A\cdot T_1^{p_1(rn)}\1_A\cdots
T_l^{p_l(rn)}\1_A\bigg\|_2\\
\geq&\lim_{N-M\to\infty}\bigg\|\frac1{N-M}\sum_{n=M}^{N-1}\1_A\prod_{j=1}^lT_j^{p_j(rn)}\E(\1_A|\K_{rat}(T_j))\bigg\|_2-
\frac{\epsilon}{2^{l+1}}\cdot2^l.
\end{align*}
Next, choose a large $r\in r_0\N$ such that
$$
\|\E(\1_A|\K_r(T_j))-\E(\1_A|\K_{rat}(T_j))\|_2\leq\frac{\epsilon}{23l}
$$
for $j=1,\ldots,l$.
So
$$
\bigg|\int \1_A\prod_{j=1}^lT_j^{p_j(rn)}\E(\1_A|\K_{rat}(T_j))d\mu-\int \1_A\prod_{j=1}^lT_j^{p_j(rn)}\E(\1_A|\K_r(T_j))d\mu\bigg|\leq\frac{\epsilon}3.
$$
Since $p_j(0)=0$, we have $T_j^{p_j(rn)}\E(\1_A|\K_r(T_j))=\E(\1_A|\K_r(T_j))$.
So
\begin{align*}
&\lim_{N-M\to\infty}\frac1{N-M}\sum_{n=M}^{N-1}\int\1_A\cdot T_1^{p_1(rn)}\1_A\cdots
T_l^{p_l(rn)}\1_A d\mu\\
\geq&\lim_{N-M\to\infty}\frac1{N-M}\sum_{n=M}^{N-1}\int\1_A\prod_{j=1}^lT_j^{p_j(rn)}\E(\1_A|\K_{rat}(T_j))d\mu-
\frac{\epsilon}{3}\\
\geq&\int \1_A\prod_{j=1}^l\E(\1_A|\K_r(T_j))d\mu-
\frac{\epsilon}{2}-\frac{\epsilon}{3}\geq \bigg(\int \1_Ad\mu\bigg)^{l+1}-\frac{5\epsilon}6,
\end{align*}
where Lemma \ref{EfX} is used in the last step. Now there exists $N_0=N_0(\epsilon)$ such that if $N-M\geq N_0$, then
$$
\frac1{N-M}\sum_{n=M}^{N-1}\mu(A\cap T_1^{-p_1(rn)}A\cdots
T_l^{p_l(rn)}A)\geq \mu(A)^{l+1}-\epsilon,
$$
i.e., $\mu(A\cap T_1^{-p_1(rn)}A\cdots
T_l^{p_l(rn)}A)\geq \mu(A)^{l+1}-\epsilon$ for some $n\in[M,N]$. Thus the gaps of 
$$
\{n:\,\mu(A\cap T_1^{-p_1(rn)}A\cdots
T_l^{p_l(rn)}A)\geq \mu(A)^{l+1}-\epsilon\}
$$
are bounded by $N_0$.
\end{proof}

\end{document}